\documentclass{amsart}
\usepackage{graphicx, amsmath, amssymb}
\newtheorem{thm}{Theorem}[section]
\newtheorem{cor}[thm]{Corollary}

\newtheorem{prop}[thm]{Proposition}
\theoremstyle{definition}
\newtheorem{defn}[thm]{Definition}
\theoremstyle{remark}
\newtheorem{rem}[thm]{Remark}
\numberwithin{equation}{section}

\newcommand{\OO}{\mathcal{O}}
\newcommand{\OX}{\mathcal{O}_{X}}
\newcommand{\EE}{\mathcal{E}}
\newcommand{\FF}{\mathcal{F}}
\newcommand{\GG}{\mathcal{G}}

\newcommand{\CC}{\mathcal{C}}

\begin{document}

\title[Ulrich Bundles]{Ulrich Bundles on Quartic Surfaces with Picard Number 1}
\author{Emre Coskun}
\address{Department of Mathematics, Middle East Technical University, Ankara, TURKEY}
\email{emcoskun@metu.edu.tr}

\thanks{Parts of this note were written during a visit to the Chennai Mathematical Institute, and during the author's stay at the Tata Institute of Fundamental Research. The author thanks both institutions for their hospitality, and Professor Vikraman Balaji for asking the questions that led to the writing of this note.}

\date{\today}
\begin{abstract}
In this note, we prove that there exist stable Ulrich bundles of every even rank on a smooth quartic surface $X \subset \mathbb{P}^3$ with Picard number 1.
\end{abstract}
\maketitle
\section{Introduction}

The classification of Ulrich bundles on projective varieties has been of interest recently. Casanellas and Hartshorne classified stable Ulrich bundles by their first Chern class over a cubic surface in \cite{CH12}. They give necessary and sufficient conditions for the existence of stable Ulrich bundles on a cubic surface in terms of the rank and the first Chern class. In \cite{CKM12-K3}, Coskun, Kulkarni and Mustopa proved the existence of simple Ulrich bundles of rank 2 with $c_1=3H$ on every smooth quartic surface in projective 3-space. In this note, we prove the existence of a stable Ulrich bundle of every even rank on a smooth quartic surface $X \subset \mathbb{P}^3$ with Picard number 1. By simple degree considerations, no Ulrich bundle of odd rank can exist on $X$.

We use Casanellas and Hartshorne's method, used in \cite{CH12}. First, we prove the existence of \emph{simple} Ulrich bundles of a given even rank. We will then compute the dimension of the modular family of simple Ulrich bundles with this rank. We will finally prove that the strictly semistable Ulrich bundles are parametrized by a space of strictly smaller dimension. This will then prove the existence of stable Ulrich bundles.

\subsection{Conventions}
\begin{itemize}
 \item We work over an algebraically closed field of characteristic 0.
 \item $X \subset \mathbb{P}^3$ denotes a smooth quartic surface with Picard number 1.
 \item All sheaves, and all cohomology groups, are over $X$.
 \item Let $\EE$ be a vector bundle. We denote $\EE \otimes \OX(H)^{\otimes t}$ by $\EE(t)$ for $t \in \mathbb{Z}$.
 \item Stability and semistability are defined in the sense of Gieseker.
 \item We denote the various cohomology groups in capital letters, e.g. $\mbox{H}^1(\EE)$ and $\mbox{Ext}^1(\EE,\FF)$; and their dimensions in lowercase letters, e.g. $\mbox{h}^1(\EE)$ and $\mbox{ext}^1(\EE,\FF)$.
\end{itemize}
\section{Generalities}
$X$ is a K3 surface, hence its canonical bundle $K_X$ is trivial. Recall that the Picard group of $X$ is free abelian of rank 1, generated by the hyperplane class $H$.

We start with the definition of an Ulrich bundle.
\begin{defn}
A vector bundle $\EE$ of rank $r$ is called an \emph{Ulrich bundle} if for a linear projection $\pi: X \to \mathbb{P}^2$ we have $\pi_* \EE \cong \OO_{\mathbb{P}^2}^{4r}$.
\end{defn}

This condition is often impractical to use. Therefore, we will give an equivalent condition for a vector bundle $\EE$ to be Ulrich. We first give a definition.
\begin{defn}
A vector bundle $\EE$ is called \emph{arithmetically Cohen-Macaulay (ACM)} if for every $t \in \mathbb{Z}$ we have $H^1(\EE(t))=0$.
\end{defn}

\begin{prop}\label{prop-ulrich}
A vector bundle $\EE$ of rank $r$ is Ulrich if and only if $\EE$ is ACM and its Hilbert polynomial is $P(t)=2r(t+2)(t+1)$.
\end{prop}
\begin{proof}
See \cite[Proposition 2.8]{CKM12-ternary}.
\end{proof}
\begin{rem}
This proposition shows that the definition of an Ulrich bundle is independent of the linear projection $\pi:X \to \mathbb{P}^2$.
\end{rem}

The following proposition gives the stability properties of Ulrich bundles.
\begin{prop}\label{prop-destabilize}
An Ulrich bundle $\EE$ is semistable. If $\EE$ is strictly semistable, then it is destabilized by an Ulrich subbundle. For an Ulrich bundle, stability and $\mu$-stability are equivalent.
\end{prop}
\begin{proof}
See \cite[Proposition 2.14 and Lemma 2.15]{CKM12-K3}. The last statement follows from \cite[Theorem 2.9 (c)]{CH12}.
\end{proof}

\begin{prop}\label{prop-23}
In the short exact sequence $0 \to \EE \to \FF \to \GG \to 0$ of coherent sheaves, if any two of $\EE$, $\FF$ and $\GG$ are Ulrich bundles, then so is the third.
\end{prop}
\begin{proof}
See \cite[Proposition 2.14]{CKM12-ternary}.
\end{proof}

\begin{rem}
The above two propositions imply the following result. Any strictly se\-mi\-stable Ulrich bundle is obtained from lower rank stable Ulrich bundles by consecutive extensions. We will use this fact in the proof of the theorem.
\end{rem}

We end this section by some numerical results.
\begin{prop}\label{prop-degree}
The degree of an Ulrich bundle of rank $r$ is $6r$.
\end{prop}
\begin{proof}
This follows from \cite[Proposition 2.10]{CKM13-delPezzo}.
\end{proof}
\begin{rem}
The \emph{slope} of a vector bundle is defined to be the degree divided by the rank. Hence, the slope of an Ulrich bundle is 6.
\end{rem}

\begin{cor}\label{cor-no_odd_rank}
On $X$, there exist no Ulrich bundles of odd rank. Therefore, any rank 2 Ulrich bundle on $X$ is stable.
\end{cor}
\begin{proof}
Indeed, the hyperplane class $H$ has degree 4. The first assertion follows from the fact that no odd multiple of 6 can be divisible by 4. The second assertion follows since any rank 2 Ulrich bundle which is not stable would have to contain an Ulrich subbundle of rank 1 by Proposition \ref{prop-destabilize}.
\end{proof}

The following result will be crucial in computing dimensions of various parameter spaces.
\begin{prop}\label{prop-chi}
Let $\EE$ and $\FF$ be Ulrich bundles of ranks $r$ and $s$ respectively, and with first Chern classes $C$ and $D$ respectively. Then
\begin{equation}\label{eqn-1}
 \chi(\EE^{\vee} \otimes \FF) = -C.D+6rs.
\end{equation}
\end{prop}
\begin{proof}
See \cite[Proposition 2.12]{CH12}.
\end{proof}
\section{The Proof of the Main Result}

\begin{thm}
There exist stable Ulrich bundles on $X$ of every even rank $2k$.
\end{thm}
\begin{rem}
Note that by Proposition \ref{prop-degree}, for an Ulrich bundle of rank $2k$, the only possibility for the first Chern class is $3kH$.
\end{rem}
\begin{proof}
We use the method of Casanellas-Hartshorne as used in \cite{CH12}. We proceed by induction on $k$. The case $k=1$ follows from \cite[Theorem 1.1]{CKM12-K3} and Corollary \ref{cor-no_odd_rank}.

Suppose that the theorem is proved for all ranks smaller than $2k$. Choose a stable Ulrich bundle $\FF_1$ of rank $2k-2$ and another stable Ulrich bundle $\FF_2$ of rank 2, such that $\FF_1 \ncong \FF_2$ if $k=2$. (We note that this is possible since the moduli space of stable rank-2 Ulrich bundles is 14-dimensional by \cite[Theorem 1.1]{CKM12-K3}.) By Equation (\ref{eqn-1}), $\mbox{ext}^1(\FF_1,\FF_2)=12(k-1) > 0$ and we can choose a non-split extension of $\FF_1$ by $\FF_2$. Therefore, by \cite[Lemma 4.2]{CH12}, this extension is simple; and by Proposition \ref{prop-23}, it is Ulrich. Hence there exist \emph{simple} Ulrich bundles of rank $2k$ with first Chern class $3kH$.

We now compute the dimension of the modular family of simple Ulrich bundles of rank $2k$. We recall that this dimension is equal to $\mbox{h}^1(\EE^{\vee} \otimes \EE)$ for a simple Ulrich bundle $\EE$ of rank $2k$. We have $\mbox{h}^0(\EE^{\vee} \otimes \EE)=\mbox{hom}(\EE, \EE)=1$ and $\mbox{h}^2(\EE^{\vee} \otimes \EE)=\mbox{h}^0(\EE^{\vee} \otimes \EE)=1$ by Serre duality. Hence by Equation (\ref{eqn-1}), we get
\begin{align*}
 \mbox{h}^1(\EE^{\vee} \otimes \EE) &= 2-\chi(\EE^{\vee} \otimes \EE) \\
 &= 2-(-(3kH)^2+6(2k)^2) \\
 &= 12k^2+2.
\end{align*}

The next step is to consider consecutive extensions of stable Ulrich bundles. Consider stable Ulrich bundles $\EE_1, \ldots, \EE_n$ with ranks $2k_1, \ldots, 2k_n$, such that $\EE_i \ncong \EE_j$ if $i \neq j$. For each $\EE_i$, the associated moduli space has dimension $12k_i^2+2$ by the calculation above; and hence the dimension of the moduli space parametrizing the choices of $\EE_1,\ldots,\EE_n$ is $(12k_1^2+2)+\cdots+(12k_n^2+2)$. We consider strictly semistable Ulrich bundles $\mathcal{E}$ of rank $2k$ constructed from $a_i$ copies of $\EE_i$, $i=1,\ldots,n$, by subsequent extensions. Then we have $k=a_1 k_1 + \cdots + a_n k_n$. Write $m=a_1+\cdots+a_n$ for the total number of bundles used. We also write $l_1,\ldots,l_m$ for the sequence of the $k_i$s, where each $k_i$ is taken $a_i$ times. We denote this collection of stable Ulrich bundles by $\CC$.

By Equation (\ref{eqn-1}), for any two Ulrich bundles $\FF$ and $\GG$ of ranks $2f$ and $2g$, we get
\begin{align*}
 \mbox{ext}^1(\FF,\GG) &= \mbox{h}^1(\FF^{\vee} \otimes \GG) \\
 &= \mbox{h}^0(\FF^{\vee} \otimes \GG) + \mbox{h}^2(\FF^{\vee} \otimes \GG) - \chi(\FF^{\vee} \otimes \GG) \\
 &= \mbox{h}^0(\FF^{\vee} \otimes \GG) + \mbox{h}^2(\FF^{\vee} \otimes \GG) + 12 fg.
\end{align*}
Now consider a collection of $a$ copies of $\EE_i$ for some fixed $1 \leq i \leq n$, and $b$ copies of $\EE_j$ for some fixed $1 \leq j \leq n$, with $i \neq j$ and $a + b \geq 2$. We can form an Ulrich bundle $\FF$ by extending a member of this collection with another member, and by extending the result with another member of the collection, and so on. In other words, we can construct a strictly semistable Ulrich bundle $\FF$ whose Jordan-H\"older factors (which are stable) are precisely this collection of $a$ copies of $\EE_i$ and $b$ copies of $\EE_j$. We then have $\mbox{h}^0(\FF^{\vee} \otimes \EE_i) \leq a$ and $\mbox{h}^0(\EE_i^{\vee} \otimes \FF) \leq a$, and by Serre duality $\mbox{h}^2(\FF^{\vee} \otimes \EE_i) \leq a$ and $\mbox{h}^2(\EE_i^{\vee} \otimes \FF) \leq a$. Using these bounds we obtain that, for Ulrich bundles $\EE$ obtained from subsequent extensions of all the $\EE_i$s in the collection $\CC$, the moduli space parametrizing them has dimension at most
\begin{align*}
 (12k_1^2+2)+\cdots+(12k_n^2+2) + 12 \sum_{i<j} l_i l_j -(m-1)+\sum_{i=1}^{n} a_i(a_i-1).
\end{align*}

Let us explain the count above. The terms $12k_i^2+2$ for $1 \leq i \leq n$ are given by the moduli count for the $\EE_i$. The other terms arise after considering the fact that the isomorphism classes of extensions of a vector bundle $\EE$ by $\FF$ are parametrized by $\mathbb{P}(\mbox{Ext}^1(\EE, \FF))$. The dimension $\mbox{ext}^1(\FF,\GG)$ for Ulrich bundles $\FF$ and $\GG$ of ranks $2f$ and $2g$ was computed above. Subtracting 1 and taking the sum, we obtain the last three terms in the above dimension count.

After writing $k=a_1 k_1 + \cdots + a_n k_n$ and expanding $k^2$, it can be shown that the moduli count above is smaller than $12k^2+2$, which is equal to the dimension of the moduli space of simple Ulrich bundles of rank $2k$. Hence, there must exist a simple Ulrich bundle of rank $2k$ which is not strictly semistable. It follows that this Ulrich bundle is stable of rank $2k$.
\end{proof}


\end{document}